\documentclass[12pt,a4paper]{article}
\usepackage[utf8x]{inputenc}
\usepackage{amsmath}
\usepackage{amsfonts}
\usepackage{amssymb}
\usepackage{amsthm}
\usepackage{mathrsfs}
\usepackage{fancyhdr}
\usepackage{graphicx}
\usepackage[usenames,x11names]{xcolor}
\usepackage{arabtex}
\usepackage{framed}
\usepackage[top=2cm,bottom=2cm,margin=2cm]{geometry}
\usepackage{cancel}
\usepackage{array}   

\usepackage[dvipdfm,pdfstartview=FitH,colorlinks=true,linkcolor=Red4,urlcolor=Red4,%
filecolor=green,citecolor=red]{hyperref}

\title{Nontrivial effective lower bounds for the least common multiple of a $q$-arithmetic progression}
\author{\sc Bakir FARHI \\
Laboratoire de Mathématiques appliquées \\
Faculté des Sciences Exactes \\
Université de Bejaia, 06000 Bejaia, Algeria \\[1mm]
\href{mailto:bakir.farhi@gmail.com}{bakir.farhi@gmail.com} \\[1mm]
\url{http://farhi.bakir.free.fr/}
}
\date{}

\def\R{{\mathbb R}}

\def\N{{\mathbb N}}
\def\Z{{\mathbb Z}}

\def\lcm{\mathrm{lcm}}
\def\gcd{\mathrm{gcd}}

\newcommand{\qbinom}[2]{{\left[\begin{array}{c}
#1 \\ #2
\end{array}\right]}_q}    

\newcommand{\qbinomsd}[2]{{[\begin{subarray}{c}
#1 \\ #2
\end{subarray}]}_q}    

\def\EMdash{\leavevmode\hbox to 10.6mm{\vrule height .63ex depth -.59ex
    width 10mm\hfill}}

\theoremstyle{plain}
\numberwithin{equation}{section}
\newtheorem{thm}{Theorem}[section]

\newtheorem{lemma}[thm]{Lemma}

\newtheorem{coll}[thm]{Corollary}

\newtheorem{thmn}{Theorem}

\begin{document}
\maketitle
\begin{abstract}
This paper is devoted to establish nontrivial effective lower bounds for the least common multiple of consecutive terms of a sequence ${(u_n)}_{n \in \mathbb{N}}$ whose general term has the form $u_n = r {[n]}_q + u_0$, where $q , r$ are positive integers and $u_0$ is a non-negative integer such that $\mathrm{gcd}(u_0 , r) = \mathrm{gcd}(u_1 , q) = 1$. For such a sequence, we show that for all positive integer $n$, we have $\mathrm{lcm}\{u_1 , u_2 , \dots , u_n\} \geq c_1 \cdot c_2^n \cdot q^{\frac{n^2}{4}}$, where $c_1$ and $c_2$ are positive constants depending only on $q , r$ and $u_0$. This can be considered as a $q$-analog of the lower bounds already obtained by the author (in 2005) and by Hong and Feng (in 2006) for the arithmetic progressions.
\end{abstract}

\section{Introduction and the main results}\label{sec1}

Throughout this paper, we let $\N^*$ denote the set $\N \setminus \{0\}$ of positive integers. For $t \in \R$, we let $\lfloor t\rfloor$ denote the integer part of $t$. We say that an integer $a$ is a multiple of a non-zero rational number $r$ if the quotient $a / r$ is an integer. The letter $q$ always denotes a positive integer; besides, it is assumed if necessary that $q \geq 2$ (this assumption is needed in §\ref{sec2.2}). Let us recall the standard notations of $q$-calculus (see e.g., \cite{kac}). For $n , k \in \N$, with $n \geq k$, we have by definition:
\begin{eqnarray*}
{[n]}_q & := & \frac{q^n - 1}{q - 1} ~~ \text{for } q \neq 1 \text{ and } {[n]}_1 := n , \\[2mm]
{[n]}_q! & := & {[n]}_q {[n - 1]}_q \cdots {[1]}_q ~~ (\text{with the convention } {[0]_q! = 1}) , 
\end{eqnarray*}
$$
\qbinom{n}{k} ~:=~ \frac{{[n]}_q!}{{[k]}_q! {[n - k]}_q!} ~=~ \frac{{[n]}_q {[n - 1]}_q \cdots {[n - k + 1]}_q}{{[k]}_q!} . ~~~~~~~~~~~~
$$

\noindent The numbers $\qbinomsd{n}{k}$ are called the $q$-binomial coefficients (or the gaussian binomial coefficients) and it is well-known that they are all positive integers (see e.g., \cite{kac}). From this last fact, we derive the important property stating that:
\begin{equation}\label{eq6}
\text{For all } a , b \in \N, \text{ the positive integer } {[a]}_q! {[b]}_q! \text{ divides the positive integer } {[a + b]}_q! .  
\end{equation}
Indeed, for $a , b \in \N$, we have $\frac{{[a + b]}_q!}{{[a]}_q! {[b]}_q!} ~=~ \qbinomsd{a + b}{a} \in \N^*$.

The study of the least common multiple of consecutive positive integers began with Chebychev's work \cite{cheb} in his attempts to prove the prime number theorem. The latter defined $\psi(n) := \log\lcm(1 , 2 , \dots , n)$ ($\forall n \geq 2$) and showed that $\frac{\psi(n)}{n}$ is bounded between two positive constants, but he failed to prove that $\psi(n) \sim_{+ \infty} n$, which is equivalent to the prime number theorem. Quite recently, Hanson \cite{han} and Nair \cite{nair} respectively obtained in simple and elegant ways that $\lcm(1 , 2 , \dots , n) \leq 3^n$ ($\forall n \in \N^*$) and $\lcm(1 , 2 , \dots , n) \geq 2^n$ ($\forall n \geq 7$). Later, the author \cite{far1,far2} obtained nontrivial effective lower bounds for the least common multiple of consecutive terms in an arithmetic progression. In particular, he proved that for any $u_0 , r , n \in \N^*$, with $\gcd(u_0 , r) = 1$, we have $\lcm(u_0 , u_0 + r , \dots , u_0 + n r) ~\geq~ u_0 (r + 1)^{n - 1}$. By developing the author's method, Hong and Feng \cite{hong1} managed to improve this lower bound to the optimal one: 
\begin{equation}\label{eq7}
\lcm(u_0 , u_0 + r , \dots , u_0 + n r) ~\geq~ u_0 (r + 1)^n ~~~~ (\forall n \in \N) ,
\end{equation}
which is already conjectured by the author \cite{far1,far2}. It is interesting to note that the method used to obtain \eqref{eq7} is based on the following fundamental theorem:
\begin{thmn}[{\cite[Theorem 2]{far2}}]\label{tf}
Let $I$ be a finite non-empty set of indices and ${(u_i)}_{i \in I}$ be a sequence of non-zero integers. Then the integer
$$
\lcm\left\{u_i ,~ i \in I\right\} \cdot \lcm\left\{\prod_{\begin{subarray}{c}
i \in I \\ i \neq j
\end{subarray}} \vert u_i - u_j\vert ,~ j \in I\right\}
$$
is a multiple of the integer $\displaystyle\prod_{i \in I} u_i$.
\end{thmn}
Furthermore, several authors obtained improvements of \eqref{eq7} for $n$ sufficiently large in terms of $u_0$ and $r$ (see e.g., \cite{hong2,kane}). Concerning the asymptotic estimates and the effective upper bounds for the least common multiple of an arithmetic progression, we can cite the work of Bateman et al. \cite{bat} and the very recent work of Bousla \cite{bous}.

In this paper, we apply and adapt the author's method \cite{far1,far2} (slightly developed by Hong and Feng \cite{hong1}) to establish nontrivial effective lower bounds for the least common multiple of consecutive terms in a sequence that we called {\it a $q$-arithmetic progression}; that is a sequence ${(u_n)}_n$ with general term has the form $u_n = r {[n]}_q + u_0$ ($\forall n \in \N$), where $r \in \N^*$, $u_0 \in \N$ and $r , u_0 , q$ satisfy some technical conditions. Our main results are the following:

\begin{thm}[The crucial result]\label{t1}
Let $q$ and $r$ be two positive integers and $u_0$ be a non-negative integer. Let ${(u_n)}_{n \in \N}$ be the sequence of natural numbers whose general term $u_n$ is given by: $u_n = r {[n]}_q + u_0$. Suppose that $\gcd(u_0 , r) = \gcd(u_1 , q) = 1$. Then, for any positive integers $n$ and $k$ such that $n \geq k$, the positive integer $\lcm\{u_k , u_{k + 1} , \dots , u_n\}$ is a multiple of the rational number $\frac{u_k u_{k + 1} \cdots u_n}{{[n - k]}_q!}$.
\end{thm}

\begin{thm}\label{t2}
In the situation of Theorem \ref{t1}, set
$$
A ~:=~ \max\left(0 ~,~ \frac{u_0 (q - 1) + 1 - r}{2 r}\right) .
$$
Then, for any positive integer $n$, we have
$$
\lcm\{u_1 , u_2 , \dots , u_n\} ~\geq~ u_1 \left(\frac{r + 1}{\sqrt{r} (A + 1)}\right)^{n - 1} q^{\frac{(n - 1) (n - 4)}{4}} .
$$
\end{thm}

\begin{thm}\label{t3}
In the situation of Theorem \ref{t1}, set
$$
B ~:=~ \max\left(r ~,~ \frac{u_0 (q - 1) + 1 - r}{2}\right) . 
$$
Then, for any positive integer $n$, we have
$$
\lcm\{u_1 , u_2 , \dots , u_n\} ~\geq~ u_1 \left(\frac{r + 1}{2 \sqrt{B}}\right)^{n - 1} q^{\frac{(n - 1) (n - 4)}{4}} .
$$
\end{thm}

Note that Theorem \ref{t1} is a $q$-analog of a result due to the author (see \cite[Théorème 2.3]{far1} or \cite[Theorem 3]{far2}). Furthermore, Theorems \ref{t2} and \ref{t3} are derived from Theorem \ref{t1} by optimizing a certain specific expression, and they can be considered as $q$-analogs of the results by the author \cite{far1,far2} and those by Hong and Feng \cite{hong1}.

From Theorems \ref{t2} and \ref{t3}, we immediately derive the two following corollaries:

\begin{coll}\label{c3}
Let $q$, $a$ and $b$ be integers such that $q \geq 2$, $a \geq 1$ and $b \geq - a$ and let ${(v_n)}_{n \in \N}$ be the sequence of natural numbers whose general term $v_n$ is given by:
$$
v_n ~=~ a q^n + b ~~~~ (\forall n \in \N) .
$$
Suppose that $\gcd(a q , b) = \gcd(a + b , q - 1) = 1$ and set
$$
A' ~:=~ \max\left(0 ~,~ \frac{b}{2 a} + \frac{1}{2 a (q - 1)}\right) .
$$
Then, for any positive integer $n$, we have
$$
\lcm\{v_1 , v_2 , \dots , v_n\} ~\geq~ (a q + b) \left(\frac{a (q - 1) + 1}{\sqrt{a (q - 1)} (A' + 1)}\right)^{n - 1} q^{\frac{(n - 1) (n - 4)}{4}} .
$$
\end{coll}

\begin{coll}\label{c4}
In the situation of Corollary \ref{c3}, set
$$
B' ~:=~ \max\left(a (q - 1) ~,~ \frac{b (q - 1) + 1}{2}\right) .
$$
Then, for any positive integer $n$, we have
$$
\lcm\{v_1 , v_2 , \dots , v_n\} ~\geq~ (a q + b) \left(\frac{a (q - 1) + 1}{2 \sqrt{B'}}\right)^{n - 1} q^{\frac{(n - 1) (n - 4)}{4}} .
$$
\end{coll}

\section{The proofs}

Throughout the following, we fix $q , r \in \N^*$ and $u_0 \in \N$ such that $\gcd(u_0 , r) = \gcd(u_1 , q) = 1$ and we let ${(u_n)}_{n \in \N}$ denote the sequence of natural numbers defined by its general term $u_n := r {[n]}_q + u_0$ ($\forall n \in \N$).

\subsection{Proof of Theorem \ref{t1}}

To prove Theorem \ref{t1}, We need the three following lemmas:

\begin{lemma}\label{l1}
For all $i , j \in \N$, we have
$$
\left\vert u_i - u_j\right\vert ~=~ r q^{\min(i , j)} {[|i - j|]}_q .
$$
\end{lemma}

\begin{proof}
Let $i , j \in \N$. Because the two sides of the equality of the lemma are both symmetric (in $i$ and $j$), we may suppose without loss of generality that $i \geq j$. Doing so, we have
\begin{eqnarray*}
\left\vert u_i - u_j\right\vert = u_i - u_j & = & \left(r {[i]}_q + u_0\right) - \left(r {[j]}_q + u_0\right) \\
& = & r \left({[i]}_q - {[j]}_q\right) \\[1mm]
& = & r \left(\frac{q^i - 1}{q - 1} - \frac{q^j - 1}{q - 1}\right) \\[1mm]
& = & r \left(\frac{q^i - q^j}{q - 1}\right) \\[1mm]
& = & r q^j \left(\frac{q^{i - j} - 1}{q - 1}\right) \\[1mm]
& = & r q^j {[i - j]}_q \\
& = & r q^{\min(i , j)} {[|i - j|]}_q ,
\end{eqnarray*}
as required. The lemma is proved.
\end{proof}

\begin{lemma}\label{l2}
For all $n \in \N$, we have
$$
\gcd(u_n , r) ~=~ 1 .
$$
If in addition $n \geq 1$, then we have
$$
\gcd(u_n , q) ~=~ 1 .
$$
\end{lemma}

\begin{proof}
Let $n \in \N$ and let us show that $\gcd(u_n , r) = 1$. This is equivalent to show that $d = 1$ is the only positive common divisor of $u_n$ and $r$. So, let $d$ be a positive common divisor of $u_n$ and $r$ and let us show that $d = 1$. The hypothesis $d | u_n$ and $d | r$ imply $d | (u_n - r {[n]}_q) = u_0$. Hence $d$ is a positive common divisor of $u_0$ and $r$. But since $\gcd(u_0 , r) = 1$, it follows that $d = 1$, as required. Consequently, we have $\gcd(u_n , r) = 1$.

Next, let $n \in \N^*$ and let us show that $\gcd(u_n , q) = 1$. Equivalently, we have to show that $d = 1$ is the only positive common divisor of $u_n$ and $q$. So, let $d$ be a positive common divisor of $u_n$ and $q$ and let us show that $d = 1$. The hypothesis $d | u_n$ and $d | q$ imply \linebreak $d | \{(r q^n + u_0 q) - (q - 1) u_n\} = r + u_0 = u_1$. So, $d$ is a positive common divisor of $u_1$ and $q$. But since $\gcd(u_1 , q) = 1$, we conclude that $d = 1$, as required. Consequently, we have $\gcd(u_n , q) = 1$. This completes the proof of the lemma. 
\end{proof}

\begin{lemma}\label{l3}
For any positive integers $n$ and $k$ such that $n \geq k$ and any $j \in \{k , k + 1 , \dots , n\}$, we have
$$
\sum_{\begin{subarray}{c}
k \leq i \leq n \\
i \neq j
\end{subarray}} \min(i , j) ~\leq~ \frac{(n - k) (n + k - 1)}{2} .
$$ 
\end{lemma}

\begin{proof}
Let $n$ and $k$ be positive integers such that $n \geq k$ and let $j \in \{k , k + 1 , \dots , n\}$. We have
\begin{eqnarray*}
\sum_{\begin{subarray}{c}
k \leq i \leq n \\
i \neq j
\end{subarray}} \min( i , j) & = & \sum_{k \leq i < j} \min(i , j) + \sum_{j < i \leq n} \min(i , j) \\
& = & \sum_{k \leq i < j} i + \sum_{j < i \leq n} j \\[1mm]
& = & \frac{(j - k) (j + k - 1)}{2} + (n - j) j \\[1mm]
& = & \frac{2 n j - j^2 - k^2 - j + k}{2} \\[1mm]
& = & \frac{(n - k) (n + k - 1) + (n - j) - (n - j)^2}{2} \\[1mm]
& \leq & \frac{(n - k) (n + k - 1)}{2}
\end{eqnarray*}
(since $n - j \leq (n - j)^2$, because $n - j \in \N$). The lemma is proved.
\end{proof}

Now, we are ready to prove the crucial theorem \ref{t1}:

\begin{proof}[Proof of Theorem \ref{t1}]
Let $n$ and $k$ be positive integers such that $n \geq k$. By applying the fundamental theorem \ref{tf} to the set of indices $I = \{k , k + 1 , \dots , n\}$ and to the sequence ${(u_i)}_{i \in I} = \{u_k , u_{k + 1} , \dots , u_n\}$, we find that the positive integer
$$
\lcm\left\{u_k , u_{k + 1} , \dots , u_n\right\} \cdot \lcm\left\{\prod_{\begin{subarray}{c}
k \leq i \leq n \\
i \neq j
\end{subarray}} \left\vert u_i - u_j\right\vert;~ j = k , \dots , n\right\}
$$
is a multiple of the positive integer $u_k u_{k + 1} \cdots u_n$. Now, let us find a simple multiple for the positive integer $\lcm\left\{\prod_{k \leq i \leq n , i \neq j} |u_i - u_j| ; j = k , \dots , n\right\}$. According to Lemma \ref{l1}, we have for any $j \in \{k , k + 1 , \dots , n\}$:
$$
\prod_{\begin{subarray}{c}
k \leq i \leq n \\
i \neq j
\end{subarray}} \left\vert u_i - u_j\right\vert ~=~ \prod_{\begin{subarray}{c}
k \leq i \leq n \\
i \neq j
\end{subarray}} \left(r q^{\min(i , j)} {[|i - j|]}_q\right) ~~~~~~~~~~~~~~~~~~~~~~~~~~~~~~~~~~~~~~~~~~~~~~~~~~~~~~~~~~~~
$$
\begin{eqnarray*}
& = & r^{n - k} q^{\sum_{\begin{subarray}{c}
k \leq i \leq n \\
i \neq j
\end{subarray}} \min(i , j)} \prod_{\begin{subarray}{c}
k \leq i \leq n \\
i \neq j
\end{subarray}} {\left[|i - j|\right]}_q \\
& = & r^{n - k} q^{\sum_{\begin{subarray}{c}
k \leq i \leq n \\
i \neq j
\end{subarray}} \min(i , j)} {[1]}_q {[2]}_q \cdots {[j - k]}_q \times {[1]}_q {[2]}_q \cdots {[n - j]}_q \\[1mm]
& = & r^{n - k} q^{\sum_{\begin{subarray}{c}
k \leq i \leq n \\
i \neq j
\end{subarray}} \min(i , j)} {[j - k]}_q! {[n - j]}_q! ,
\end{eqnarray*}
which divides (according to Lemma \ref{l3} and Property \eqref{eq6}) the positive integer
$$
r^{n - k} q^{\frac{(n - k)(n + k - 1)}{2}} {[n - k]}_q! .
$$
Consequently, the positive integer $\lcm\{\prod_{k \leq i \leq n , i \neq j} |u_i - u_j| ; j = k , \dots , n\}$ divides the positive integer $r^{n - k} q^{\frac{(n - k) (n + k - 1)}{2}} {[n - k]}_q!$. It follows (according to what obtained at the beginning of this proof) that the positive integer $u_k u_{k + 1} \cdots u_n$ divides the positive integer \linebreak $r^{n - k} q^{\frac{(n - k) (n + k - 1)}{2}} {[n - k]}_q! \, \lcm\{u_k , u_{k + 1} , \dots , u_n\}$. Next, since (according to Lemma \ref{l2}) the \linebreak integers $u_i$ ($i \geq 1$) are all coprime with $r$ and $q$ then the product $u_k u_{k + 1} \cdots u_n$ is coprime with $r^{n - k} q^{\frac{(n - k) (n + k - 1)}{2}}$, which concludes (according to the Gauss lemma) that $u_k u_{k + 1} \cdots u_n$ \linebreak divides ${[n - k]}_q! \, \lcm\{u_k , u_{k + 1} , \dots , u_n\}$. Equivalently, the positive integer $\lcm\{u_k , u_{k + 1} , \dots , u_n\}$ is a multiple of the rational number $\frac{u_k u_{k + 1} \cdots u_n}{{[n - k]}_q!}$. This achieves the proof. 
\end{proof}

\subsection{Proofs of Theorems \ref{t2} and \ref{t3} and their corollaries}\label{sec2.2}

To deduce Theorems \ref{t2} and \ref{t3} from Theorem \ref{t1}, we need some additional preparations. Since, for $q = 1$, Theorems \ref{t2} and \ref{t3} are immediate consequences of \eqref{eq7}, we may suppose for the sequel that $\boldsymbol{q \geq 2}$. Next, we naturally extend the definition of $u_n$ to negative indices $n$ and we define for all $n , k \in \Z$ such that $n \geq k$:
$$
C_{n , k} ~:=~ \frac{u_k u_{k + 1} \cdots u_n}{{[n - k]}_q!} .
$$
Furthermore, for a given positive integer $n$, the problem of determining the positive integer $k \leq n$ which maximizes $C_{n , k}$ leads us to introduce the function $f : \R \rightarrow \R$, defined by:
$$
f(x) ~:=~ q^{x - 1} \left(r q^{x - 1} + u_0 (q - 1) + 1 - r\right) ~~~~~~~~~~ (\forall x \in \R) .
$$
It is immediate that $f$ increases, tends to $0$ as $x$ tends to $(- \infty)$ and satisfies, for all $n \in \N^*$, the property:
$$
\forall k \in \Z :~~ k > n ~\Rightarrow~ f(k) > q^n .
$$
For a given positive integer $n$, these properties ensure the existence of a largest $k_n \in \Z$ satisfying $f(k_n) \leq q^n$, and show, in addition, that $k_n \leq n$. From the increase of $f$ and the definition of $k_n$ ($n \in \N^*$), we derive that:
\begin{equation}\label{eq1}
\forall k \in \Z :~~ k \leq k_n ~\Longleftrightarrow~ f(k) \leq q^n .
\end{equation}
Now, since for any $n \in \N^*$ and any $k \in \Z$, we have
\begin{eqnarray*}
f(k) \leq q^n & \Longleftrightarrow & q^{k - 1} \left(r q^{k - 1} + u_0 (q - 1) + 1 - r\right) \leq q^n \\
& \Longleftrightarrow & r q^{k - 1} + u_0 (q - 1) + 1 - r \leq q^{n - k + 1} \\
& \Longleftrightarrow & \frac{q^{n - k + 1} - 1}{q - 1} \geq r \frac{q^{k - 1} - 1}{q - 1} + u_0 \\
& \Longleftrightarrow & {[n - k + 1]}_q \geq u_{k - 1} ,
\end{eqnarray*}
then Property \eqref{eq1} is equivalent to:
\begin{equation}\label{eq2}
\forall k \in \Z :~~ k \leq k_n ~\Longleftrightarrow~ {[n - k + 1]}_q \geq u_{k - 1} .
\end{equation}
For a given positive integer $n$, we set
$$
\ell_n ~:=~ \max(1 , k_n) .
$$
Since $k_n \leq n$, we have that: $\ell_n \in \{1 , 2 , \dots , n\}$.

Next, it is immediate that $f$ satisfies the following inequality:
\begin{equation}\label{eq3}
f(x - 1) ~\leq~ \frac{1}{q} f(x) ~~~~~~~~~~ (\forall x \in \R) .
\end{equation}

For a fixed $n \in \N^*$, the following lemmas aim to maximize the quantity $C_{n , k}$ ($1 \leq k \leq n$) appearing in Theorem \ref{t1}. Precisely, we shall determine two simple upper bounds for $\max_{1 \leq k \leq n} C_{n , k}$ from which we derive our theorems \ref{t2} and \ref{t3}.

\begin{lemma}\label{l5}
Let $n$ be a fixed positive integer. The sequence ${(C_{n , k})}_{k \in \Z , k \leq n}$ is non-decreasing until $k = k_n$ then it decreases. So, it reaches its maximal value at $k = k_n$.
\end{lemma}

\begin{proof}
For any $k \in \Z$, with $k \leq n$, we have
\begin{eqnarray*}
C_{n , k} \geq C_{n , k - 1} & \Longleftrightarrow & \frac{C_{n , k}}{C_{n , k - 1}} \geq 1 \\[2mm]
& \Longleftrightarrow & \frac{u_k u_{k + 1} \cdots u_n}{{[n - k]}_q!} \Big{/} \frac{u_{k - 1} u_k \cdots u_n}{{[n - k + 1]}_q!} \geq 1 \\[2mm]
& \Longleftrightarrow & \frac{{[n - k + 1]}_q}{u_{k - 1}} \geq 1 \\[2mm]
& \Longleftrightarrow & {[n - k + 1]}_q \geq u_{k - 1} \\
& \Longleftrightarrow & k \leq k_n ~~~~~~~~~~ (\text{according to \eqref{eq2}}) ,
\end{eqnarray*}
which concludes to the result of the lemma.
\end{proof}

From the last lemma, we obviously derive the following:

\begin{lemma}\label{l6}
Let $n$ be a fixed positive integer. Then the sequence ${(C_{n , k})}_{1 \leq k \leq n}$ reaches its maximal value at $k = \ell_n$. \hfill $\square$
\end{lemma}

If $n \in \N^*$ is fixed, we have from Lemma \ref{l6} above that $\max_{1 \leq k \leq n} C_{n , k} = C_{n , \ell_n}$; however, the exact value of $C_{n , \ell_n}$ (in terms of $n , q , r , u_0$) is complicated. The lemmas below provide studies of the sequences ${(k_n)}_n$, ${(\ell_n)}_n$ and ${(C_{n , \ell_n})}_n$ in order to find a good lower bound for $C_{n , \ell_n}$ which has a simple expression in terms of $n , q , r , u_0$.

\begin{lemma}\label{l7}
For all positive integer $n$, we have
$$
k_n ~\leq~ k_{n + 1} ~\leq~ k_n + 1 .
$$
In other words, we have
$$
k_{n + 1} \in \left\{k_n ~,~ k_n + 1\right\} .
$$
\end{lemma}

\begin{proof}
Let $n$ be a fixed positive integer. By definition of the integer $k_n$, we have
$$
f(k_n) ~\leq~ q^n ~\leq~ q^{n + 1} ,
$$
which implies (by definition of the integer $k_{n + 1}$) that:
$$
k_{n + 1} ~\geq~ k_n .
$$
On the other hand, we have (according to \eqref{eq3} and to the definition of the integer $k_{n + 1}$):
$$
f(k_{n + 1} - 1) ~\leq~ \frac{1}{q} f(k_{n + 1}) ~\leq~ \frac{1}{q} \, q^{n + 1} = q^n ,
$$
which implies (by definition of the integer $k_n$) that:
$$
k_n ~\geq~ k_{n + 1} - 1 ;
$$
that is
$$
k_{n + 1} ~\leq~ k_n + 1 .
$$
This completes the proof of the lemma.
\end{proof}

\begin{lemma}\label{l8}
For all positive integer $n$, we have
$$
\ell_{n + 1} \in \left\{\ell_n ~,~ \ell_n + 1\right\} .
$$
In addition, in the case when $\ell_{n + 1} = \ell_n + 1$, we have $\ell_n = k_n$ and $\ell_{n + 1} = k_{n + 1} = k_n + 1$.
\end{lemma}

\begin{proof}
Let $n$ be a fixed positive integer. By Lemma \ref{l7}, we have that:
$$
k_n ~\leq~ k_{n + 1} ~\leq~ k_n + 1 .
$$
Hence
$$
\max(1 , k_n) ~\leq~ \max(1 , k_{n + 1}) ~\leq~ \max(1 , k_n + 1) = \max(0 , k_n) + 1 \leq \max(1 , k_n) + 1 ;
$$
therefore
$$
\ell_n ~\leq~ \ell_{n + 1} ~\leq~ \ell_n + 1 .
$$
This confirms the first part of the lemma.

Now, let us show the second part of the lemma. So, suppose that $\ell_{n + 1} = \ell_n + 1$ and show that $\ell_n = k_n$ and $\ell_{n + 1} = k_{n + 1} = k_n + 1$. Since $\ell_n = \max(1 , k_n) \geq 1$ and $\ell_{n + 1} = \ell_n + 1$ then $\ell_{n + 1} \geq 2$. This implies that $\ell_{n + 1} \neq 1$; thus $\ell_{n + 1} = k_{n + 1}$ (since $\ell_{n + 1} = \max(1 , k_{n + 1}) \in \{1 , k_{n + 1}\}$). Using this and Lemma \ref{l7} above, we derive that: $\ell_n = \ell_{n + 1} - 1 = k_{n + 1} - 1 \leq (k_n + 1) - 1 = k_n$; that is $\ell_n \leq k_n$. But since $\ell_n = \max(1 , k_n) \geq k_n$, we conclude that $\ell_n = k_n$. This completes the proof of the second part of the lemma and achieves this proof. 
\end{proof}

\begin{lemma}\label{l9}
For all positive integer $n$, we have
$$
C_{n + 1 , \ell_{n + 1}} ~\geq~ (r + 1) q^{\ell_n - 1} C_{n , \ell_n} .
$$
\end{lemma}

\begin{proof}
Let $n$ be a fixed positive integer. By Lemma \ref{l8}, we have that $\ell_{n + 1} \in \{\ell_n , \ell_{n + 1}\}$. So, we have to distinguish two cases:\\[1mm]
\underline{1\textsuperscript{st} case:} (if $\ell_{n + 1} = \ell_n$) \\[1mm]
In this case, we have
\begin{align}
C_{n + 1 , \ell_{n + 1}} ~=~ C_{n + 1 , \ell_n} ~=~ \frac{u_{\ell_n} u_{\ell_n + 1} \cdots u_n u_{n + 1}}{{[n + 1 - \ell_n]}_q!} &~=~ \frac{u_{\ell_n} u_{\ell_n + 1} \cdots u_n}{{[n - \ell_n]}_q!} \cdot \frac{u_{n + 1}}{{[n + 1 - \ell_n]}_q} \notag \\
&~=~ C_{n , \ell_n} \cdot \frac{u_{n + 1}}{{[n + 1 - \ell_n]}_q} . \label{eq4}
\end{align}
Next, we have
\begin{eqnarray*}
u_{n + 1} - (r + 1) q^{\ell_n - 1} {[n + 1 - \ell_n]}_q & = & r {[n + 1]}_q + u_0 - (r + 1) q^{\ell_n - 1} \left(\frac{q^{n + 1 - \ell_n} - 1}{q - 1}\right) \\[1mm]
& = & r \left(\frac{q^{n + 1} - 1}{q - 1}\right) + u_0 - (r + 1) \left(\frac{q^n - q^{\ell_n - 1}}{q - 1}\right) \\[1mm]
& = & \frac{r (q^{n + 1} - 1) + u_0 (q - 1) - (r + 1) (q^n - q^{\ell_n - 1})}{q - 1} \\[1mm]
& = & \frac{r q^{n + 1} - (r + 1) q^n + (r + 1) q^{\ell_n - 1} - r + u_0 (q - 1)}{q - 1} \\[1mm]
& = & \frac{\big(r (q - 1) - 1\big) q^n + [(r + 1) q^{\ell_n - 1} - r] + u_0 (q - 1)}{q - 1} \\
& \geq & 0
\end{eqnarray*}
(since $q \geq 2$, $r \geq 1$, $u_0 \geq 0$ and $\ell_n \geq 1$). Thus
$$
\frac{u_{n + 1}}{{[n + 1 - \ell_n]}_q} ~\geq~ (r + 1) q^{\ell_n - 1} .
$$
By reporting this into \eqref{eq4}, we get
$$
C_{n + 1 , \ell_{n + 1}} ~\geq~ (r + 1) q^{\ell_n - 1} C_{n , \ell_n} ,
$$
as required. \\[1mm]
\underline{2\textsuperscript{nd} case:} (if $\ell_{n + 1} = \ell_n + 1$) \\[1mm]
In this case, we have (according to Lemma \ref{l8}): $\ell_n = k_n$ and $\ell_{n + 1} = k_{n + 1} = k_n + 1$. Thus, we have
\begin{equation}
C_{n + 1 , \ell_{n + 1}} ~=~ C_{n + 1 , k_n + 1} ~=~ \frac{u_{k_n + 1} u_{k_n + 2} \cdots u_n u_{n + 1}}{{[n - k_n]}_q!} ~=~ C_{n , k_n} \cdot \frac{u_{n + 1}}{u_{k_n}} ~=~ C_{n , \ell_n} \cdot \frac{u_{n + 1}}{u_{k_n}} . \label{eq5}
\end{equation}
Next, according to the inequality of the right-hand side of \eqref{eq2} (applied for $(n + 1)$ instead of $n$ and $k_{n + 1}$ instead of $k$), we have (since $k_{n + 1} \leq k_{n + 1}$):
$$
u_{k_n} ~=~ u_{k_{n + 1} - 1} ~\leq~ {\left[(n + 1) - k_{n + 1} + 1\right]}_q ~=~ {[n - k_n + 1]}_q .
$$
Hence:
\begin{eqnarray*}
u_{n + 1} - (r + 1) q^{\ell_n - 1} u_{k_n} & = & u_{n + 1} - (r + 1) q^{k_n - 1} u_{k_n} \\
& \geq & u_{n + 1} - (r + 1) q^{k_n - 1} {[n - k_n + 1]}_q \\
& = & r \left(\frac{q^{n + 1} - 1}{q - 1}\right) + u_0 - (r + 1) q^{k_n - 1} \left(\frac{q^{n - k_n + 1} - 1}{q - 1}\right) \\[1mm]
& = & \frac{r (q^{n + 1} - 1) + u_0 (q - 1) - (r + 1) (q^n - q^{k_n - 1})}{q - 1} \\
& = & \frac{(r (q - 1) - 1) q^n + u_0 (q - 1) + (r + 1) q^{k_n - 1} - r}{q - 1} \\
& \geq & 0
\end{eqnarray*}
(since $q \geq 2$, $r \geq 1$, $u_0 \geq 0$ and $k_n = \ell_n \geq 1$). Thus
$$
\frac{u_{n + 1}}{u_{k_n}} ~\geq~ (r + 1) q^{\ell_n - 1} .
$$
By reporting this into \eqref{eq5}, we get
$$
C_{n + 1 , \ell_{n + 1}} ~\geq~ (r + 1) q^{\ell_n - 1} C_{n , \ell_n} ,
$$
as required. The proof of the lemma is complete.
\end{proof}

By induction, we derive from Lemma \ref{l9} above the following:

\begin{coll}\label{c1}
For all positive integer $n$, we have
$$
C_{n , \ell_n} ~\geq~ u_1 (r + 1)^{n - 1} q^{\sum_{i = 1}^{n - 1} (\ell_i - 1)} .
$$
\end{coll}

\begin{proof}
Let $n$ be a positive integer. From Lemma \ref{l9}, we have
\[
C_{n , \ell_n} ~=~ C_{1 , \ell_1} \prod_{i = 1}^{n - 1} \frac{C_{i + 1 , \ell_{i + 1}}}{C_{i , \ell_i}} ~\geq~ C_{1 , \ell_1} \prod_{i = 1}^{n - 1} \left\{(r + 1) q^{\ell_i - 1}\right\} ~=~ C_{1 , \ell_1} (r + 1)^{n - 1} q^{\sum_{i = 1}^{n - 1} (\ell_i - 1)} .
\]
Next, since $k_1 \leq 1$, we have $\ell_1 = \max(1 , k_1) = 1$; hence $C_{1 , \ell_1} = C_{1 , 1} = \frac{u_1}{{[0]}_q!} = u_1$. Consequently, we have
\[
C_{n , \ell_n} ~\geq~ u_1 (r + 1)^{n - 1} q^{\sum_{i = 1}^{n - 1} (\ell_i - 1)} ,
\]
as required. The corollary is proved.
\end{proof}

From Theorem \ref{t1} and Corollary \ref{c1} above, we immediately deduce the following:

\begin{coll}\label{c2}
For all positive integer $n$, we have
$$
\lcm\left\{u_1 , u_2 , \dots , u_n\right\} ~\geq~ u_1 (r + 1)^{n - 1} q^{\sum_{i = 1}^{n - 1} (\ell_i - 1)} .
$$
\end{coll}

\begin{proof}
Let $n$ be a fixed positive integer. Since the positive integer $\lcm\{u_1 , u_2 , \dots , u_n\}$ is obviously a multiple of the positive integer $\lcm\{u_{\ell_n} , u_{\ell_n + 1} , \dots , u_n\}$, which is a multiple of the rational number $\frac{u_{\ell_n} u_{\ell_n + 1} \cdots u_n}{{[n - \ell_n]}_q!} = C_{n , \ell_n}$ (according to Theorem \ref{t1}), then we have
$$
\lcm\left\{u_1 , u_2 , \dots , u_n\right\} ~\geq~ C_{n , \ell_n} .
$$
The result of the corollary then follows from Corollary \ref{c1}. The proof is achieved.
\end{proof}

\noindent\textbf{Remark.} If we allow to take $q = 1$ in Corollary \ref{c2}, then we exactly obtain the result of Hong and Feng \cite{hong1} (recalled in \eqref{eq7}).

Now, in order to derive from Corollary \ref{c2} above an explicit lower bound for \linebreak $\lcm\{u_1 , u_2 , \dots , u_n\}$ ($n \geq 1$), it remains to bound from below the $\ell_i$'s in terms of $n , q , r$ and $u_0$. We just give here two ways to bound from below the $\ell_i$'s, but there are certainly other ways (perhaps more intelligent) to do this. We have the following lemmas:
  
\begin{lemma}\label{l12}
Let
$$
A ~:=~ \max\left(0 ~,~ \frac{u_0 (q - 1) + 1 - r}{2 r}\right) .
$$
Then, for all positive integer $n$, we have
$$
\ell_n ~>~ \frac{1}{2} \left(n - \frac{\log{r} + 2 \log(A + 1)}{\log{q}}\right) .
$$
\end{lemma}

\begin{proof}
Let $n$ be a fixed positive integer. Since the inequality of the lemma is obvious for $n \leq \frac{\log{r} + 2 \log(A + 1)}{\log{q}}$, we may assume for the sequel that $n > \frac{\log{r} + 2 \log(A + 1)}{\log{q}}$. Now, for any $x \geq 1$, we have
\begin{eqnarray*}
f(x) & := & q^{x - 1} \left(r q^{x - 1} + u_0 (q - 1) + 1 - r\right) \\[1mm]
& = & r \left\{\left(q^{x - 1} + \frac{u_0 (q - 1) + 1 - r}{2 r}\right)^2 - \left(\frac{u_0 (q - 1) + 1 - r}{2 r}\right)^2\right\} \\
& \leq & r \left(q^{x - 1} + \frac{u_0 (q - 1) + 1 - r}{2 r}\right)^2 \\
& \leq & r \left(q^{x - 1} + A\right)^2 \\
& \leq & r \left(q^{x - 1} + A q^{x - 1}\right)^2 \\
& = & r (A + 1)^2 q^{2 (x - 1)} .
\end{eqnarray*}
By applying this for
$$
x_0 ~:=~ \frac{1}{2} \left(n - \frac{\log{r} + 2 \log(A + 1)}{\log{q}}\right) + 1
$$
(which is $> 1$ according to our assumption $n > \frac{\log{r} + 2 \log(A + 1)}{\log{q}}$), we get
$$
f(x_0) ~\leq~ r (A + 1)^2 q^{n - \frac{\log{r} + 2 \log(A + 1)}{\log{q}}} ~=~ q^n .
$$
Then, since $f$ is increasing and $\lfloor x_0\rfloor \leq x_0$, we derive that:
$$
f(\lfloor x_0\rfloor) ~\leq~ f(x_0) ~\leq~ q^n ,
$$
which implies (according to the definition of $k_n$) that:
$$
k_n ~\geq~ \lfloor x_0\rfloor ~>~ x_0 - 1 .
$$
Hence:
$$
\ell_n ~:=~ \max(1 , k_n) ~\geq~ k_n ~>~ x_0 - 1 ,
$$
that is
$$
\ell_n ~>~ \frac{1}{2} \left(n - \frac{\log{r} + 2 \log(A + 1)}{\log{q}}\right) ,
$$
as required. The lemma is proved.
\end{proof}

\begin{lemma}\label{l13}
Let
$$
B ~:=~ \max\left(r ~,~ \frac{u_0 (q - 1) + 1 - r}{2}\right) .
$$
Then, for all positive integer $n$, we have
$$
\ell_n ~>~ \frac{1}{2} \left(n - \frac{\log(4 B)}{\log{q}}\right) .
$$
\end{lemma}

\begin{proof}
Let $n$ be a fixed positive integer. Since the inequality of the lemma is obvious for $n \leq \frac{\log(4 B)}{\log{q}}$, we may assume for the sequel that $n > \frac{\log(4 B)}{\log{q}}$. Now, for any $x \geq 1$, we have
\begin{eqnarray*}
f(x) & := & q^{x - 1} \left(r q^{x - 1} + u_0 (q - 1) + 1 - r\right) \\
& \leq & q^{x - 1} \left(B q^{x - 1} + 2 B\right) \\
& < & B \left(q^{x - 1} + 1\right)^2 \\
& \leq & B \left(2 q^{x - 1}\right)^2 \\
& = & 4 B q^{2 (x - 1)} .
\end{eqnarray*}
By applying this for
$$
x_1 ~:=~ \frac{1}{2} \left(n - \frac{\log(4 B)}{\log{q}}\right) + 1
$$
(which is $> 1$ according to our assumption $n > \frac{\log(4 B)}{\log{q}}$), we get
$$
f(x_1) ~\leq~ 4 B q^{n - \frac{\log(4 B)}{\log{q}}} ~=~ q^n .
$$
Then, since $f$ is increasing and $\lfloor x_1\rfloor \leq x_1$, we derive that:
$$
f(\lfloor x_1\rfloor) ~\leq~ f(x_1) ~\leq~ q^n ,
$$
which implies (according to the definition of $k_n$) that:
$$
k_n ~\geq~ \lfloor x_1\rfloor ~>~ x_1 - 1 ~=~ \frac{1}{2} \left(n - \frac{\log(4 B)}{\log{q}}\right) .
$$
Hence
$$
\ell_n ~:=~ \max(1 , k_n) ~\geq~ k_n ~>~ \frac{1}{2} \left(n - \frac{\log(4 B)}{\log{q}}\right) , 
$$
as required. The lemma is proved.
\end{proof}

We are now ready to prove Theorems \ref{t2} and \ref{t3} announced in §\ref{sec1}.

\begin{proof}[Proof of Theorem \ref{t2}]
By using successively Corollary \ref{c2} and Lemma \ref{l12}, we have for all $n \in \N^*$:
\begin{eqnarray*}
\lcm\left\{u_1 , u_2 , \dots , u_n\right\} & \geq & u_1 (r + 1)^{n - 1} q^{\sum_{i = 1}^{n - 1} (\ell_i - 1)} \\[1mm]
 & \geq & u_1 (r + 1)^{n - 1} q^{\frac{(n - 1) (n - 4)}{4} - \frac{1}{2} \frac{\log{r} + 2 \log(A + 1)}{\log{q}} (n - 1)} \\[1mm]
 & = & u_1 \left(\frac{r + 1}{\sqrt{r} (A + 1)}\right)^{n - 1} q^{\frac{(n - 1) (n - 4)}{4}} ,
\end{eqnarray*} 
as required.
\end{proof}

\begin{proof}[Proof of Theorem \ref{t3}]
By using successively Corollary \ref{c2} and Lemma \ref{l13}, we have for all $n \in \N^*$:
\begin{eqnarray*}
\lcm\left\{u_1 , u_2 , \dots , u_n\right\} & \geq & u_1 (r + 1)^{n - 1} q^{\sum_{i = 1}^{n - 1} (\ell_i - 1)} \\[1mm]
 & \geq & u_1 (r + 1)^{n - 1} q^{\frac{(n - 1) (n - 4)}{4} - \frac{1}{2} \frac{\log(4 B)}{\log{q}} (n - 1)} \\[1mm]
 & = & u_1 \left(\frac{r + 1}{2 \sqrt{B}}\right)^{n - 1} q^{\frac{(n - 1) (n - 4)}{4}} ,
\end{eqnarray*} 
as required.
\end{proof}

\begin{proof}[Proof of Corollary \ref{c3}]
It suffices to remark that $v_n = a (q - 1) {[n]}_q + a + b$ ($\forall n \in \N$) and then to apply Theorem \ref{t2} for the sequence ${(v_n)}_{n \in \N}$. We just specify that the imposed conditions $\gcd(a q , b) = \gcd(a + b , q - 1) = 1$ guarantee the conditions $\gcd(v_0 , r) = \gcd(v_1 , q) = 1$ required in Theorem \ref{t2} (with $r := a (q - 1)$). 
\end{proof}

\begin{proof}[Proof of Corollary \ref{c4}]
We simply apply Theorem \ref{t3} for the sequence ${(v_n)}_{n \in \N}$, after noticing that its general term can be written as: $v_n = a (q - 1) {[n]}_q + a + b$.
\end{proof}

\section{Numerical examples and remarks}

By applying our main results, we get for example the following nontrivial effective estimates:
\begin{itemize}
\item $\lcm\{2^1 - 1 , 2^2 - 1 , \dots , 2^n - 1\} \geq 2^{\frac{n (n - 1)}{4}}$ \hfill ($\forall n \geq 1$) \\
(Apply Theorem \ref{t2} for $u_n = {[n]}_2 = 2^n - 1$).
\item $\lcm\{2^1 + 1 , 2^2 + 1 , \dots , 2^n + 1\} \geq 3 \cdot 2^{\frac{(n - 1) (n - 4)}{4}}$ \hfill ($\forall n \geq 1$) \\
(Apply one of the two corollaries \ref{c3} or \ref{c4} for $v_n = 2^n + 1$).
\item $\lcm\{3^1 + 1 , 3^2 + 1 , \dots , 3^n + 1\} \geq 4 \cdot 3^{\frac{(n - 1) (n - 4)}{4}}$ \hfill ($\forall n \geq 1$) \\
(Remark that $\lcm\{3^1 + 1 , 3^2 + 1 , \dots , 3^n + 1\} = 2 \, \lcm\{\frac{3^1 + 1}{2} , \frac{3^2 + 1}{2} , \dots , \frac{3^n + 1}{2}\}$ and apply one of the two theorems \ref{t2} or \ref{t3} for $u_n = {[n]}_3 + 1 = \frac{3^n + 1}{2}$).
\end{itemize}

\noindent\textbf{Remarks.} 
\begin{enumerate}
\item Theorems \ref{t2} and \ref{t3} are incomparable in the sense that there are situations where Theorem \ref{t2} is stronger than Theorem \ref{t3} and other situations where we have the converse. For example, it is easy to verify that if $u_0 (q - 1) + 1 - r \leq 0$ then Theorem \ref{t2} is stronger than Theorem \ref{t3}, while if $u_0 (q - 1) + 1 - 3 r > 0$ then Theorem \ref{t3} is stronger than Theorem \ref{t2}.
\item By refining the arguments of bounding from below the $\ell_i$'s (that is the arguments of the proofs of Lemmas \ref{l12} and \ref{l13}), it is perhaps possible to obtain a lower bound for $\lcm\{u_1 , u_2 , \dots , u_n\}$ ($n \geq 1$) of the form:
$$
\lcm\{u_1 , u_2 , \dots , u_n\} ~\geq~ c \left(\frac{r + 1}{\sqrt{r}}\right)^{n - 1} q^{\frac{(n - 1) (n - 4)}{4}} ,
$$
where $c$ is a positive constant depending only on $q$, $r$ and $u_0$. It appears that this is the best that can be expected from this method!
\item It is remarkable that our lower bounds of $\lcm\{u_1 , u_2 , \dots , u_n\}$, for the considered sequences ${(u_n)}_n$, are quite close to $\sqrt{u_1 u_2 \cdots u_n}$. More precisely, we can easily deduce from our main results that in the same context, we have $\lcm\{u_1 , u_2 , \dots , u_n\} \geq c_3 c_4^n \sqrt{u_1 u_2 \cdots u_n}$, for some suitable positive constants $c_3$ and $c_4$, depending only on $q$, $r$ and $u_0$.
\item There is something in common between our results and the recent result by Bousla and Farhi \cite{bousfar} providing effective bounds for $\lcm(U_1 , U_2 , \dots , U_n)$, when ${(U_n)}_{n \in \N}$ is a particular Lucas sequence; precisely, when ${(U_n)}_n$ is recursively defined by: $U_0 = 0$, $U_1 = 1$ and $U_{n + 2} = P U_{n + 1} - Q U_n$ ($\forall n \in \N$) for some $P , Q \in \Z^*$, with $P^2 - 4 Q > 0$ and $\gcd(P , Q) = 1$. Indeed, if we take $P = q + 1$ and $Q = q$ (for some integer $q \geq 2$), we obtain that $U_n = {[n]}_q$ and the Bousla-Farhi lower bound then gives:
$$
\lcm\left({[1]}_q , {[2]}_q , \dots , {[n]}_q\right) ~\geq~ q^{\frac{n^2}{4} - \frac{n}{2} - 1} ~~~~ (\forall n \geq 1) ,
$$
which is almost the same as what obtained in this paper.
\end{enumerate}

\end{document}